\documentclass[11pt,a4paper]{amsart}
\usepackage{amssymb}
\usepackage{amsmath}

\usepackage{xcolor,wasysym}
\newtheorem{theorem}{Theorem}[section]
\newtheorem{lemma}{Lemma}[section]
\newtheorem{corollary}{Corollary}[section]
\newtheorem{definition}{Definition}[section]

\theoremstyle{remark}
\newtheorem{remark}{Remark}[section]
\newtheorem{remarks}{Remarks}[section]
\newtheorem{example}{Example}[section]
\numberwithin{equation}{section}
\begin{document}

\title[Estimate for Initial MacLaurin Coefficients \dots]{Estimate for Initial MacLaurin Coefficients of Certain Subclasses of Bi-univalent Functions}

\author[P. Goswami]{Pranay Goswami}
\address{School of Liberal Studies, Ambedkar University Delhi, Delhi-110006, India}
\email{pranaygoswami83@gmail.com}

\author[B. S. Alkahtani]{Badr S. Alkahtani}
\address{Mathematics Department, College of Science, King Saud University, P.O.Box 1142, Riyadh 11989, Saudi Arabia}
\email{alhaghog@gmail.com}

\author[T. Bulboac\u{a}]{Teodor Bulboac\u{a}}
\address{Faculty of Mathematics and Computer Science, Babe\c{s}-Bolyai University, 400084 Cluj-Napoca, Romania}
\email{bulboaca@math.ubbcluj.ro}

\date{}
\begin{abstract}
In this paper, estimates for second and third MacLaurin coefficients of certain subclasses of bi-univalent functions in the open unit disk defined by convolution are determined, and certain special cases are also indicated. The main result extends and improve a recent one obtained by Srivastava et al.
\end{abstract}

\dedicatory{This paper is dedicated to Prof. Ravi P. Agarwal}



\maketitle
\section{Introduction and definitions}

Let $\mathcal{A}$ be the class of functions $f$ which are analytic in the open unit disk $\mathbb{D}=\{z\in\mathbb{C}:\left\vert z\right\vert<1\}$ and normalized by the conditions $f(0)=0$ and $f'(0)=1$. The Koebe one-quarter theorem \cite{duren} ensures that the image of $\mathbb{D}$ under every univalent function $f\in\mathcal{A}$ contains the disk with the center in the origin and the radius $1/4$. Thus, every univalent function $f\in\mathcal{A}$ has an inverse $f^{-1}:f(\mathbb{D})\rightarrow\mathbb{D}$, satisfying $f^{-1}(f(z))=z$, $z\in\mathbb{D}$, and
\[f\left(f^{-1}(w)\right)=w,\;|w|<r_0(f),\;r_0(f)\geq\frac{1}{4}.\]

Moreover, it is easy to see that the inverse function has the series expansion of the form
\begin{equation}\label{expinv}
f^{-1}(w)=w-a_2w^2+\left(2a_2^2-a_3\right)w^3-\left(5a_2^3-5a_2a_3+a_4\right)w^4+\dots,\;w\in f(\mathbb{D}).
\end{equation}

A function $f\in\mathcal{A}$ is said to be {\em bi-univalent}, if both $f$ and $f^{-1}$ are univalent in $\mathbb{D}$, in the sense that $f^{-1}$ has a univalent analytic continuation to $\mathbb{D}$, and we denote by $\sigma$ this class of bi-univalent functions.

In \cite{pad} the authors defined the classes of functions $\mathcal{P}_{m}(\beta)$ as follows: let $\mathcal{P}_{m}(\beta)$, with $m\geq2$ and $0\leq\beta<1$, denote the class of univalent analytic functions $P$, normalized with $P(0)=1$, and satisfying
\[\int_0^{2\pi}\left|\frac{\operatorname{Re}P(z)-\beta}{1-\beta}\right|\operatorname{d}\theta\leq m\pi,\]
where $z=re^{i\theta}\in\mathbb{D}$.

For $\beta=0$, we denote $\mathcal{P}_m:=\mathcal{P}_m(0)$, hence the class $\mathcal{P}_m$ represents the class of functions $p$ analytic in $\mathbb{D}$, normalized with $p(0)=1$, and having the representation
\begin{equation}\label{eq}
p(z)=\int\limits_0^{2\pi}\frac{1- ze^{it}}{1+ze^{it}}\operatorname{d}\mu(t),
\end{equation}
where $\mu$ is a real-valued function with bounded variation, which satisfies
\begin{equation}\label{eqP}
\int_0^{2\pi}d\mu(t)=2\pi\quad\text{and}\quad\int_0^{2\pi}\left|d\mu(t)\right|\leq m,\;m\geq2.
\end{equation}
Clearly, $\mathcal{P}:=\mathcal{P}_2$ is the well-known class of {\em Carath\'eodory functions}, i.e. the normalized functions with positive real part in the open unit disk $\mathbb{D}$ .

Lewin \cite{Lewin} investigated the class $\sigma$ of bi-univalent functions and obtained the bound for the second coefficient. Brannan and Taha \cite{bran} considered certain subclasses of bi-univalent functions, similar to the familiar subclasses of univalent functions consisting of strongly starlike, starlike and convex functions. They introduced the concept of bi-starlike functions and the bi-convex functions, and obtained estimates for the initial coefficients. Recently, Ali et al. \cite{ali3}, Srivastava et al. \cite{sri}, Frasin and Aouf \cite{frasin}, Goyal and Goswami \cite{goyal} and many others have introduced and investigated subclasses of bi-univalent functions and obtained bounds for the initial coefficients. Motivated by work of Srivastava et al. \cite{sri}, we introduce a new subclass of bi-univalent functions, as follows.

For the functions $f,\;g\in\mathcal{A}$ given by
\[f(z)=z+\sum_{n=2}^{\infty}a_{n}z^{n},\quad g(z)=z+\sum_{n=2}^{\infty}b_{n}z^{n},\;z\in\mathbb{D},\]
we recall the {\em Hadamard (or convolution) product} of $f$ and $g$, defined by
\[\left(f\ast g\right)(z)=z+\sum_{n=2}^{\infty}a_{n}b_{n}z^{n},\;z\in\mathbb{D}.\]

\begin{definition}\label{def1.1}
For a given function $k\in\sigma$, a function $f\in\sigma$ is said to be in the class $\mathcal{BR}^k(m;\beta)$, with $m\geq2$ and $0\leq\beta<1$, if the following conditions are satisfied
\begin{eqnarray*}
&&\frac{(f*k)(z)}{z}\in\mathcal{P}_{m}(\beta),\\
&&\frac{(g*k)(w)}{w}\in\mathcal{P}_{m}(\beta),
\end{eqnarray*}
where $g=f^{-1}$.
\end{definition}

\begin{remark}\label{rem1.1}
Taking $k(z)=z/(1-z)^2$ and $m =2$ in the Definition \ref{def1.1} we obtain the class $\mathcal{B}(\beta):=\mathcal{BR}^{z/(1-z)^2}(2;\beta)$ studied by Srivastava et al. \cite[Definition 2]{sri}.
\end{remark}

\begin{definition}\label{def1.2}
For a given function $k\in\sigma$ and a number $\alpha\in\mathbb{C}$, a function $f\in\sigma$ is said to be in the class $\mathcal{BV}^k(m;\alpha,\beta)$, with $m\geq2$ and $0\leq\beta<1$, if the following conditions are satisfied
\begin{eqnarray*}
&&(1-\alpha)\frac{z(f*k)'(z)}{(f*k)(z)}+\alpha\left(1+\frac{z(f*k)''(z)}{(f*k)'(z)}\right)
\in\mathcal{P}_{m}(\beta),\\
&&(1-\alpha)\frac{w(g*k)'(w)}{(g*k)(w)}+\alpha\left(1+\frac{w(g*k)''(w)}{(g*k)'(w)}\right)\in\mathcal{P}_{m}(\beta),
\end{eqnarray*}
where $g=f^{-1}$.
\end{definition}

\begin{remarks}\label{rem1.2}
$(i)$ Taking $\alpha=0$ and $\alpha=1$ in the above class $\mathcal{BV}^k(m;\alpha,\beta)$ we obtain the classes $\mathcal{S}_{m}^k(\beta):=\mathcal{BV}^k(m;0,\beta)$ and $\mathcal{C}_{m}^k(\beta):=\mathcal{BV}^k(m;1,\beta)$, respectively.

$(ii)$ Moreover, if we take $k(z)=z/(1-z)$ and $m=2$, the classes $\mathcal{S}_{m}^k(\beta)$ and $\mathcal{C}_{m}^k(\beta)$ reduces to the well-known classes of {\em bi-starlike} and {\em bi-convex functions}, respectively (see also \cite{bran}).
\end{remarks}

The object of the paper is to find estimates for the coefficients $a_2$ and $a_3$ for functions in the subclass $\mathcal{BR}^k(m;\beta)$ and $\mathcal{BV}^k(m;\alpha,\beta)$, and these bounds are obtained by employing the techniques used earlier by Srivastava et al. \cite{sri}.
\section{Main results}

In order to prove our main result for the functions $f\in\mathcal{BR}^k(m;\beta)$, first we will prove the following lemma:

\begin{lemma}\label{lem2.1}
Let the function $\Phi(z)=1+\sum\limits_{n=1}^\infty{h_n}{z^n}$, $z\in\mathbb{D}$, such that $\Phi\in\mathcal{P}_m(\beta)$. Then,
\[\left|h_n\right|\leq m(1-\beta),\;n\geq1.\]
\end{lemma}

\begin{proof}
From \eqref{eq} and \eqref{eqP}, like in \cite{pad} and \cite{noor}, we can see that if $p\in\mathcal{P}_m$, then
\begin{equation}\label{eq1}
p(z)=\left(\frac{m}{4}+\frac{1}{2}\right)p_1(z)-\left(\frac{m}{4}-\frac{1}{2}\right)p_2(z),
\end{equation}
where $p_1,p_2\in\mathcal{P}$.

Further, if $p(z)=1+\sum\limits_{n=1}^{\infty}p_nz^n$, $z\in\mathbb{D}$, where $p_1(z)=1+\sum\limits_{n=1}^{\infty}p_n^{(1)}z^n$ and $p_2(z)=1+\sum\limits_{n=1}^{\infty}p_n^{(2)}z^n$ for all $z\in\mathbb{D}$, comparing the coefficients of both sides of \eqref{eq1} we get
\[p_n=\left(\frac{m}{4}+\frac{1}{2}\right)p_n^{(1)}-\left(\frac{m}{4}-\frac{1}{2}\right)p_n^{(2)},\;n\geq1.\]
Since $p_1,p_2\in\mathcal{P}$, where $\mathcal{P}$ is the class of Carath\'eodory functions, it is well-known that  $|p_n^{(1)}|\leq2$ and $|p_n^{(2)}|\leq2$ for all $n\geq1$, and thus
\begin{equation}\label{eq2}
\left|p_n\right|\leq\left(\frac{m}{4}+\frac{1}{2}\right)\left|p_n^{(1)}\right|+
\left(\frac{m}{4}-\frac{1}{2}\right)\left|p_n^{(2)}\right|\leq2\left(\frac{m}{4}+\frac{1}{2}\right)+
2\left(\frac{m}{4}-\frac{1}{2}\right)=m,\;n\geq1.
\end{equation}

Now, the proof of this lemma is straight forward, if we write
\[\Phi(z)=(1-\beta )p(z)+\beta,\quad\text{where}\quad p(z)=1+\sum\limits_{n=1}^\infty p_nz^n\in\mathcal{P}_m.\]
Then,
\[\Phi(z)=1+(1-\beta)\sum\limits_{n=1}^\infty{p_n}z^n,\;z\in\mathbb{D},\]
which gives
$$h_n=(1-\beta)p_n,\;n\geq1,$$
and using the inequality \eqref{eq2} we obtain the desired result.
\end{proof}

\begin{theorem}\label{thm2.1}
Let $f(z)=z+\sum\limits_{n=2}^{\infty}a_{n}z^{n}$ be in the class $\mathcal{BR}^k(m;\beta)$, where $k\in\sigma$ has the form $k(z)=z+\sum\limits_{n=2}^{\infty}k_{n}z^{n}$. If $k_2,\,k_3\neq0$, then
\[|a_2|\leq\min\left\{\sqrt{\frac{m(1-\beta)}{\left|k_3\right|}};\;\frac{m(1-\beta)}{\left|k_2\right|}\right\},
\quad\left|a_3\right|\leq\frac{m(1-\beta)}{\left|k_3\right|},\quad\text{and}
\quad\left|2a_2^2-a_3\right|\leq\frac{m(1-\beta)}{\left|k_3\right|}.\]
\end{theorem}

\begin{proof}
Since $f\in\mathcal{BR}^k(m;\beta)$, from the Definition \ref{def1.1} we have
\begin{equation}\label{2.2}
\frac{(f*k)(z)}{z}=p(z)
\end{equation}
and
\begin{equation}\label{2.3}
\frac{(g*k)(w)}{w}=q(w),
\end{equation}
where $p,\;q\in\mathcal{P}_m(\beta)$ and $g=f^{-1}$. Using the fact that the functions $p$ and $q$ have the following Taylor expansions
\begin{eqnarray}
&&p(z)=1+p_1z+p_2z^2+p_3z^3+\dots,\;z\in\mathbb{D},\label{2.4}\\
&&q(w)=1+q_1w+q_2w^2+q_3w^3+\dots,\;w\in\mathbb{D},\label{2.5}
\end{eqnarray}
and equating the coefficients in \eqref{2.2} and \eqref{2.3}, from \eqref{expinv} we get
\begin{equation}\label{2.6}
k_2a_2=p_1,
\end{equation}
\begin{equation}\label{2.7}
k_3{a_3}=p_2,
\end{equation}
and
\begin{equation}\label{2.9}
k_3\left(2a_2^2-a_3\right)=q_2.
\end{equation}

Since $p,\;q\in\mathcal{P}_m(\beta)$, according to Lemma \ref{lem2.1}, the next inequalities hold:
\begin{eqnarray}
&&\left|p_k\right|\leq m(1-\beta),\;k\geq1,\label{2.13}\\
&&\left|q_k\right|\leq m(1-\beta),\;k\geq1,\label{2.14}
\end{eqnarray}
and thus, from \eqref{2.7} and \eqref{2.9}, by using the inequalities \eqref{2.13} and \eqref{2.14}, we obtain
\[\left|a_2\right|^2\leq\frac{\left|q_2\right|+\left|p_2\right|}{2\left|k_3\right|}\leq
\frac{m(1-\beta)}{\left|k_3\right|},\]
which gives
\begin{equation}\label{2.16}
\left|a_2\right|\leq\sqrt{\frac{m(1-\beta)}{\left|k_3\right|}}.
\end{equation}

From \eqref{2.6}, by using \eqref{2.13} we obtain immediately that
\[|a_2|=\left|\frac{p_1}{k_2}\right|\leq\frac{m(1-\beta)}{\left|k_2\right|},\]
and combining this with the inequality \eqref{2.16}, the first inequality of the conclusion is proved.

According to \eqref{2.7}, from \eqref{2.13} we easily obtain
\[|a_3|=\left|\frac{p_2}{k_3}\right|\leq\frac{m(1-\beta)}{\left|k_3\right|},\]
and from \eqref{2.9}, by using \eqref{2.13} and \eqref{2.14} we finally deduce
\[\left|2a_2^2-a_3\right|=\left|\frac{q_2}{k_3}\right|\leq\frac{m(1-\beta)}{\left|k_3\right|},\]
which completes our proof.
\end{proof}

Setting $\beta=0$ in Theorem \ref{thm2.1} we get the following special case:

\begin{corollary}\label{cor2.1}
Let $f(z)=z+\sum\limits_{n=2}^{\infty}a_{n}z^{n}$ be in the class $\mathcal{BR}^k(m;0)$, where $k\in\sigma$ has the form $k(z)=z+\sum\limits_{n=2}^{\infty}k_{n}z^{n}$. If $k_2,\,k_3\neq0$, then
\[\left|a_2\right|\leq\min\left\{\sqrt{\frac{m}{\left|k_3\right|}};\;\frac{m}{\left|k_2\right|}\right\},
\quad\left|a_3\right|\leq\frac{m}{\left|k_3\right|},\quad\text{and}\quad
\left|2a_2^2-a_3\right|\leq\frac{m}{\left|k_3\right|}.\]
\end{corollary}

For $k(z)=z/(1-z)^2$ the above corollary reduces to the next result:

\begin{example}\label{ex2.1}
If $f(z)=z+\sum\limits_{n=2}^{\infty}a_{n}z^{n}$ is in the class $\mathcal{BR}^{z/(1-z)^2}(m;0)$, then
\[\left|a_2\right|\leq\sqrt{\frac{m}{3}},\quad\left|a_3\right|\leq\frac{m}{3},\quad\text{and}\quad
\left|2a_2^2-a_3\right|\leq\frac{m}{3}.\]
\end{example}

Taking $k(z)=z/(1-z)$ in Corollary \ref{cor2.1}, we get:
\begin{example}\label{ex2.2}
If $f(z)=z+\sum\limits_{n=2}^{\infty}a_{n}z^{n}$ is in the class $\mathcal{BR}^{z/(1-z)}(m;0)$, then
\[\left|a_2\right|\leq\sqrt m,\quad\left|a_3\right|\leq m,\quad\text{and}\quad\left|2a_2^2-a_3\right|\leq m.\]
\end{example}

If we put $k(z)=z/(1-z)^2$ in Theorem \ref{thm2.1}, we deduce the next corollary:

\begin{corollary}\label{cor2.2}
If $f(z)=z+\sum\limits_{n=2}^{\infty}a_{n}z^{n}$ is in the class $\mathcal{B}(\beta)$, then
\[\left|a_2\right|\leq\left\{
\begin{array}{lll}
\sqrt{\frac{2(1-\beta)}{3}},&\text{if}&0\leq\beta\leq\frac{1}{3},\\[1em]
1-\beta,&\text{if}&\frac{1}{3}<\beta<1,
\end{array}
\right.\quad\left|a_3\right|\leq\frac{2(1-\beta)}{3},\quad\text{and}\quad
\left|2a_2^2-a_3\right|\leq\frac{2(1-\beta)}{3}.\]
\end{corollary}

\begin{remark}
For the special case $\frac{1}{3}<\beta<1$, the above first inequality, and the second one for all $0\leq\beta<1$, improve the estimates given by Srivastava et al. in \cite[Theorem 2]{sri}.
\end{remark}

\begin{theorem}\label{thm2.2}
Let $f(z)=z+\sum\limits_{n=2}^{\infty}a_{n}z^{n}$ be in the class $\mathcal{BV}^k(m;\alpha,\beta)$, with $\alpha\in\mathbb{C}\setminus\{-1\}$, where $k\in\sigma$ has the form $k(z)=z+\sum\limits_{n=2}^{\infty}k_{n}z^{n}$. If $k_2,\,k_3\neq0$ and
$$2(1+2\alpha)k_3-(1+3\alpha)k_2^2\neq0,$$
then
\[\left|a_2\right|\leq\min\left\{\sqrt{\dfrac{m(1-\beta)}{\left|2(1+2\alpha)k_3-(1+3\alpha)k_2^2\right|}};
\dfrac{m(1-\beta)}{|1+\alpha|\left|k_2\right|}\right\},\]
and
\begin{eqnarray*}
\left|a_3\right|\leq\min\left\{\dfrac{m(1-\beta)}{\left|2(1+2\alpha)k_3-(1+3\alpha)k_2^2\right|}+
\dfrac{m(1-\beta)}{2\left|1+2\alpha\right|\left|k_3\right|};
\dfrac{m(1-\beta)}{2\left|1+2\alpha\right|\left|k_3\right|}
\left(1+\dfrac{m(1-\beta)\left|1+3\alpha\right|}{\left|1+\alpha\right|^2}\right);\right.\\
\left.\dfrac{m(1-\beta)}{2\left|1+2\alpha\right|\left|k_3\right|}
\left(1+\dfrac{m(1-\beta)\left|4(1+2\alpha)k_3-(1+3\alpha)k_2^2\right|}
{\left|k_2\right|^2\left|1+\alpha\right|^2}\right)\right\},\quad\text{whenever}\quad
\alpha\in\mathbb{C}\setminus\left\{-\dfrac{1}{2}\right\}.
\end{eqnarray*}
\end{theorem}

\begin{proof}
If $f\in\mathcal{BV}^k(m; \alpha,\beta)$, according to the Definition \ref{def1.2} we have
\[(1-\alpha)\frac{z(f*k)'(z)}{(f*k)(z)}+\alpha\left(1+\frac{z(f*k)''(z)}{(f*k)'(z)}\right)=p(z)\]
and
\[(1-\alpha)\frac{w(g*k)'(w)}{(g*k)(w)}+\alpha\left(1+\frac{w(g*k)''(w)}{(g*k)'(w)}\right)=q(w),\]
where $p,\;q\in\mathcal{P}_m(\beta)$ and $g=f^{-1}$. Since
\begin{eqnarray*}
&&(1-\alpha)\dfrac{z(f*k)'(z)}{(f*k)(z)}+\alpha\left(1+\dfrac{z(f*k)''(z)}{(f*k)'(z)}\right)=\\
&&1+(1+\alpha)a_2 k_2z+\left[2(1+2\alpha)a_3k_3-(1+3\alpha)a_2^2k_2^2\right]z^2+\dots\;z\in\mathbb{D},
\end{eqnarray*}
and according to \eqref{expinv}
\begin{eqnarray*}
&&(1-\alpha)\dfrac{z(g*k)'(w)}{(g*k)(w)}+\alpha\left({1+\dfrac{z(g*k)''(w)}{(g*k)'(w)}}\right)=\\
&&1-(1+\alpha)a_2k_2w+\Big\{\left[4(1+2\alpha)k_3-(1+3\alpha)k_2^2\right]a_2^2-2(1+2\alpha)a_3k_3\Big\}w^2+\dots,
\;w\in\mathbb{D},
\end{eqnarray*}
from \eqref{2.4} and \eqref{2.5} combined with the above two expansion formulas, it follows that
\begin{eqnarray}
&&(1+\alpha){a_2k_2}=p_1,\label{2.30}\\
&&2(1+2\alpha){a_3k_3}-(1+3\alpha)a_2^2 k_2^2=p_2,\label{2.31}
\end{eqnarray}
and
\begin{equation}\label{2.33}
\left[4(1+2\alpha)k_3-(1+3\alpha)k_2^2\right]a_2^2-2(1+2\alpha)a_3k_3=q_2.
\end{equation}
Now, from \eqref{2.31} and \eqref{2.33} we deduce that
\begin{equation}\label{A}
a_2^2=\frac{p_2+q_2}{4(1+2\alpha)k_3-2(1+3\alpha)k_2^2},\quad\text{whenever}\quad 2(1+2\alpha)k_3-(1+3\alpha)k_2^2\neq0,
\end{equation}
and
\[4(1+2\alpha)k_3\left(a_3-a_2^2\right)=p_2-q_2.\]
Using \eqref{A} in the above relation, we obtain
\begin{eqnarray}
&a_3=\dfrac{p_2+q_2}{4(1+2\alpha)k_3-2(1+3\alpha)k_2^2}+\dfrac{p_2-q_2}{4(1+2\alpha)k_3},
\quad\text{whenever}\label{B}\\
&2(1+2\alpha)k_3-(1+3\alpha)k_2^2\neq0,\;\alpha\in\mathbb{C}\setminus\left\{-\dfrac{1}{2}\right\}.\nonumber
\end{eqnarray}
From \eqref{2.30} and \eqref{2.31} we get
\begin{equation}\label{2.38} a_3=\frac{1}{2(1+2\alpha)k_3}\left[p_2+\frac{1+3\alpha}{(1+\alpha)^2}p_1^2\right],
\quad\text{for}\quad\alpha\in\mathbb{C}\setminus\left\{-1;-\frac{1}{2}\right\},
\end{equation}
while from \eqref{2.30} and \eqref{2.33} we deduce that
\begin{equation}\label{2.39}
a_3=\frac{1}{2(1+2\alpha)k_3}
\left[-q_2+\frac{4(1+2\alpha)k_3-(1+3\alpha)k_2^2}{k_2^2(1+\alpha)^2}p_1^2\right],
\quad\text{for}\quad\alpha\in\mathbb{C}\setminus\left\{-1;-\frac{1}{2}\right\}.
\end{equation}
Combining \eqref{2.30} and \eqref{A} for the computation of the upper-bound of $\left|a_2\right|$, and \eqref{B}, \eqref{2.38} and \eqref{2.39} for the computation of $\left|a_3\right|$, by using Lemma \ref{lem2.1} we easily find the estimates of our theorem.
\end{proof}

Taking $\alpha=0$ and $\alpha=1$ in Theorem \ref{thm2.2} we obtain the following two special cases, respectively:

\begin{corollary}\label{cor2.3}
Let $f(z)=z+\sum\limits_{n=2}^{\infty}a_{n}z^{n}$ be in the class $\mathcal{S}_{m}^k(\beta)$, where $k\in\sigma$ has the form $k(z)=z+\sum\limits_{n=2}^{\infty}k_{n}z^{n}$. If $k_2,\,k_3\neq0$ and
\[2k_3-k_2^2\neq0,\]
then
\[\left|a_2\right|\leq\min\left\{\sqrt{\frac{m(1-\beta)}{\left|2k_3-k_2^2\right|}};
\frac{m(1-\beta)}{\left|k_2\right|}\right\},\]
and
\begin{eqnarray*}
\left|a_3\right|\leq\min\left\{\dfrac{m(1-\beta)}{\left|2k_3-k_2^2\right|}+\dfrac{m(1-\beta)}{2\left|k_3\right|};
\dfrac{m(1-\beta)\left(1+m(1-\beta)\right)}{2\left|k_3\right|};\right.\\
\left.\dfrac{m(1-\beta)}{2\left|k_3\right|}
\left(1+\dfrac{m(1-\beta)\left|4k_3-k_2^2\right|}{\left|k_2\right|^2}\right)\right\}.
\end{eqnarray*}
\end{corollary}

\begin{corollary}\label{cor2.4}
Let $f(z)=z+\sum\limits_{n=2}^{\infty}a_{n}z^{n}$ be in the class $\mathcal{C}_{m}^k(\beta)$, where $k\in\sigma$ has the form $k(z)=z+\sum\limits_{n=2}^{\infty}k_{n}z^{n}$. If $k_2,\,k_3\neq0$ and
\[3k_3-2k_2^2\neq0,\]
then
\[\left|a_2\right|\leq\min\left\{\sqrt{\frac{m(1-\beta)}{\left|6k_3-4k_2^2\right|}};
\frac{m(1-\beta)}{2\left|k_2\right|}\right\},\]
and
\begin{eqnarray*}
|a_3|\leq\min\left\{\dfrac{m(1-\beta)}{\left|6k_3-4k_2^2\right|}+\dfrac{m(1-\beta)}{6\left|k_3\right|};
\dfrac{m(1-\beta)\left(1+m(1-\beta)\right)}{6\left|k_3\right|};\right.\\
\left.\dfrac{m(1-\beta)}{6\left|k_3\right|}
\left(1+\dfrac{m(1-\beta)\left|3k_3-k_2^2\right|}{\left|k_2\right|^2}\right)\right\}.
\end{eqnarray*}
\end{corollary}

\bigskip

{\bf Conflict of Interests.} The authors declare that they have no conflict of interests regarding the publication of this paper.

\end{document}